\documentclass[11pt]{amsart}

\usepackage{amssymb,amsthm,amsmath,url,relsize,etoolbox,cite,nameref,enumerate,fancyhdr}
\usepackage{hyperref}

\theoremstyle{plain}
\newtheorem{theorem}[subsection]{Theorem}

\newtheorem{lemma}[subsection]{Lemma}

\newtheorem{proposition}[subsection]{Proposition}

\newtheorem{corollary}[subsection]{Corollary}

\newtheorem{definition}[subsection]{Definition}

\newtheorem{remark}[subsection]{Remark}

\DeclareMathOperator{\degree}{deg}
\DeclareMathOperator{\modulus}{mod}

\renewcommand{\Re}{\operatorname{Re}}

\newcommand*\charinv[1]{\overline{#1}}

\newcommand{\lsqrbra}{[}
\newcommand{\rsqrbra}{]}

\begin{document}

\title{The Fourth Moment of Derivatives of Dirichlet $L$-functions in Function Fields.}
\author{J. C. Andrade}
\author{M. Yiasemides}
\date{\today}
\address{Department of Mathematics, University of Exeter, Exeter, EX4 4QF, UK}
\email{j.c.andrade@exeter.ac.uk}
\email{my298@exeter.ac.uk}


\maketitle

\begin{abstract}
\noindent We obtain the asymptotic main term of moments of arbitrary derivatives of $L$-functions in the function field setting. Specifically, the first, second, and mixed fourth moments. The average is taken over all non-trivial characters of a prime modulus $Q \in \mathbb{F}_q [t]$, and the asymptotic limit is as $\degree Q \longrightarrow \infty$. This extends the work of Tamam who obtained the asymptotic main term of low moments of $L$-functions, without derivatives, in the function field setting. It also expands on the work of Conrey, Rubinstein, and Snaith who cojectured, using random matrix theory, the asymptotic main term of any even moment of the derivative of the Riemann zeta-function in the number field setting.
\end{abstract}

\tableofcontents

\section{Introduction}

The moments of families of $L$-functions are part of an important area of research in analytic number theory. They are interlinked with the Lindel\"{o}f hypothesis and, while the asymptotic behaviour of the second and fourth moments are understood, higher moments have been resistant to breakthroughs for almost 100 years. \\

In 1916, Hardy and Littlewood \cite{ContrTheoryRZFDistPrimes_HL1916} proved that
\begin{align*}
\int_{0}^{T} \bigg\lvert \zeta \Big( \frac{1}{2} + it \Big) \big\rvert^2 \mathrm{d} t
\sim T \log T
\end{align*}
as $T \longrightarrow \infty$. In 1926, Ingham \cite{MeanValueTheoRMZ_Ingham1926} expanded on this by proving that
\begin{align*}
\int_{0}^{T} \bigg\lvert \zeta \Big( \frac{1}{2} + it \Big) \big\rvert^4 \mathrm{d} t
\sim \frac{1}{2 \pi} T \log^4 T 
\end{align*}
as $T \longrightarrow \infty$. Further results have since been obtained for the second and fourth moments, such as lower-order terms in the asymptotic expansions, but only conjectures have been obtained for higher moments. In 2000, Keating and Snaith \cite{RMTZeta_KeatingSnaith2000} conjecured, using random matrix theory, the asymptotic main term for all even moments. \\

Similar results and conjectures hold for families of $L$-functions, and we can work in the function field setting in addition to the classical number field setting. In this paper we look at the first, second and mixed fourth moments of arbitrary derivatives of $L$-functions in the function field setting, where we average over non-trivial Dirichlet characters of a prime modulus $Q \in \mathbb{F}_q [t]$, and we consider the asymptotic behaviour as $\degree Q \longrightarrow \infty$. \\

This expands the work of Tamam \cite{FourthMoment_Tamam} who obtained similar results, but without considering derivatives. There is an error in a part of her proof, but we have addressed that in this paper (see Remark \ref{Remark, Addressing Tamam Error}). Analogies can also be seen between our results and the work of Conrey, Rubinstein, and Snaith \cite{MomDerivCharPolyAppRZF_ConreyRubinsteinSnaith} who use random matrix theory to conjecture the asymptotic main terms of 
\begin{align*}
\int_{t=0}^{T} \Big\lvert \zeta ' \Big( \frac{1}{2}+it \Big) \Big\rvert^{2k} \mathrm{d}t 
\end{align*}
for all non-negative integers $k$. \\

\textbf{Acknowledgements:} The first author is grateful to the Leverhulme Trust (RPG-2017-320) for the support through the research project grant ``Moments of $L$-functions in Function Fields and Random Matrix Theory". The second author is grateful for an EPSRC Standard Research Studentship (DTP).


\section{Notation and Results}

For integer prime powers $q$ we have a finite field of order $q$, denoted by $\mathbb{F}_q$. The polynomial ring over this finite field is denoted by $\mathbb{F}_q [t]$, but because we are working with general $q$ we we can simply write $\mathcal{A}$. The subset of monic polynomials is denoted by $\mathcal{M}$. The degree of a polynomial is the standard definition, although we do not define it for the zero polynomial. Hence, the range $\degree A \leq n$, for any non-negative integer $n$, does not include the case $A=0$. For $A \in \mathcal{A} \backslash \{ 0 \}$ we define the norm of $A$ as $\lvert A \rvert := q^{\degree A}$, and for $A = 0$ we define $\lvert 0 \rvert := 0$. \\

Generally, we reserve upper-case letters for elements of $\mathcal{A}$, and the letters $P$ and $Q$ are reserved for prime polynomials. Note that primality and ireducibility are equivalent as $\mathcal{A}$ is a unique fatorisation domain. In this paper, the term ``prime" is taken to mean ``monic prime". We denote the set of monic primes in $\mathcal{A}$ by $\mathcal{P}$. For $A,B \in \mathcal{A}$ we denote the highest common factor and lowest common multiple by $(A,B)$ and $[A,B]$, respectively. \\

For a subset $\mathcal{S} \subseteq \mathcal{A}$ we define, for all non-negative integers $n$, $\mathcal{S}_n := \{ A \in \mathcal{S} : \degree A = n \}$. We identify $\mathcal{A}_0$ with $\mathbb{F}_q \backslash \{ 0 \} = {\mathbb{F}_q}^*$.

\begin{definition}[Dirichlet Character]
Let $R \in \mathcal{M}$. A \emph{Dirichlet character} on $\mathcal{A}$ of modulus $R$ is a function $\chi : \mathcal{A} \longrightarrow \mathbb{C}^*$ satisfying the folowing properties for all $A,B \in \mathcal{A}$:
\begin{enumerate}
\item $\chi (A) = \chi (B)$ if $A \equiv B (\modulus R)$; \label{Dir Char Def Modulus Point}
\item $\chi (AB) = \chi (A) \chi (B)$;
\item $\chi (A) = 0$ if and only if $(A,R) \neq 1$.
\end{enumerate}
\end{definition}

We denote a sum over all characters of modulus $R$ by $\sum_{\chi \modulus R}$. Also, note that we can view $\chi $ as a function on $\mathcal{A} / R \mathcal{A}$, which follows naturally from point \ref{Dir Char Def Modulus Point} above. This will allow us to use expressions such as $\chi (A^{-1})$ when $A \in (\mathcal{A} / R \mathcal{A})^*$. \\

We say $\chi$ is the \emph{trivial character} of modulus $R$ if $\chi (A) = 1$ for all $A \in \mathcal{A}$ satisfying $(A,R) =1$, and we denote such characters by $\chi_0$ (the dependence on the modulus $R$ is not shown with this notation, but when used it will be clear what the associated modulus is). We define the \emph{even characters} to be those characters $\chi$ satisfying $\chi (a) = 1$ for all $a \in \mathbb{F}_q^*$. Otherwise, we say that the character is \emph{odd}. It can be shown that there are $\phi (R)$ characters of modulus $R$ and $\frac{\phi (R)}{q-1}$ even characters of modulus $R$, where $\phi$ is the totient function. 

\begin{definition}[Dirichlet $L$-function]
Let $\chi$ be a Dirichlet character of modulus $R$. We define the associated \emph{Dirichlet $L$-function} as follows: For all $\ Re (s) > 1$,
\begin{align*}
L (s , \chi ) := \sum_{A \in \mathcal{M}} \frac{ \chi (A)}{\lvert A \rvert^s} .
\end{align*}
This has an analytic continution to $\mathbb{C}$ for non-trivial characters, and to $\mathbb{C} \backslash \{ 1 + \frac{2n \pi i}{\log q} : n \in \mathbb{Z} \}$ for trivial characters.
\end{definition}

\begin{definition}[Riemann Zeta-function in $\mathbb{F}_q \lsqrbra T \rsqrbra$]
When $\chi$ is the Dirichlet character of modulus $1$, then the associated Dirichlet $L$-function is simply the Riemann zeta-function on $\mathcal{A}$, namely
\begin{align*}
\zeta_{A} (s) := \sum_{A \in \mathcal{M}} \frac{1}{\lvert A \rvert^s}.
\end{align*}
\end{definition}

One may wish to note that in this paper, when expressing asymptotic behaviour, we remove all dependencies that we can from the implied constants. For example, if an implied constant can be taken to be $\frac{1}{q-1}$, then we will view this as independent of $q$ because $\frac{1}{q-1} \leq 1$ for all prime powers $q$. \\

We prove the following results.

\begin{theorem} \label{first_moment_theorem}
For all positive integers $k$, we have that
\begin{align*}
\frac{1}{\phi (Q)} \sum_{\substack{\chi \modulus Q \\ \chi \neq \chi_0}} L^{(k)} \Big( \frac{1}{2} , \chi  \Big)
= &\frac{-(-\log q)^k}{q^{\frac{1}{2}}-1} \frac{ (\degree Q)^k}{\lvert Q \rvert^{\frac{1}{2}}}
+ O_k \bigg( (\log q)^k \frac{ (\degree Q)^{k-1}}{\lvert Q \rvert^{\frac{1}{2}}} \bigg)
\end{align*}
as $\degree Q \longrightarrow \infty$ with $Q$ being prime. 
\end{theorem}

\begin{theorem} \label{second_moment_theorem}
For all positive integers $k$ we have that
\begin{align*}
\frac{1}{\phi (Q)} \sum_{\substack{\chi \modulus Q \\ \chi \neq \chi_0}} \bigg\lvert L^{(k)} \bigg( \frac{1}{2} , \chi \bigg) \bigg\rvert^2
= \frac{(\log q)^{2k}}{2k+1} (\degree Q)^{2k+1} + O \Big( ( \log q )^{2k} (\degree Q)^{2k}\Big)
\end{align*}
as $\degree Q \longrightarrow \infty$ with $Q$ being prime.
\end{theorem}

\begin{theorem} \label{fourth_moment_theorem}
For all non-negative integers $k,l$ we have that
\begin{align*}
&\frac{1}{\phi (Q)} \frac{1}{(\log q)^{2k+2l}} \sum_{\substack{\chi \modulus Q \\ \chi \neq \chi_0}} \Big\lvert L^{(k)} \Big( \frac{1}{2} , \chi \Big) \Big\rvert^{2} \Big\lvert L^{(l)} \Big( \frac{1}{2} , \chi \Big) \Big\rvert^{2} \\
&= (1-q^{-1}) (\degree Q)^{2k+2l+4} \int_{\substack{a_1 , a_2 , a_3 , a_4 \geq 0 \\ 2a_1 + a_3 + a_4 < 1 \\ 2a_2 + a_3 + a_4 < 1}}  f_k \big( a_1 + a_3 , a_1 + a_4 , 1 \big) f_l \big( a_2 + a_4 , a_2 + a_3 , 1 \big) \mathrm{d} a_1 \mathrm{d} a_2 \mathrm{d} a_3 \mathrm{d} a_4 \\
& \quad + O_{k,l} \Big( (\degree Q)^{2k+2l+\frac{7}{2}} \Big) 
\end{align*}
as $\degree Q \longrightarrow \infty$ with $Q$ being prime, where for all non-negative inetegers $i$ we define
\begin{align*}
f_{i} \big( x, y, z \big) = &x^i y^i + (z - x)^i (z - y)^i .
\end{align*}
\end{theorem}

These three results are extensions of Tamam's work \cite{FourthMoment_Tamam}, where she proves the three theorems above for the cases where $k,l=0$. In the number field setting, Conrey, Rubinstein, and Snaith \cite{MomDerivCharPolyAppRZF_ConreyRubinsteinSnaith} conjectured using random matrix theory that
\begin{align*}
\frac{1}{T} \int_{t=0}^{T} \Big\lvert \zeta ' \Big( \frac{1}{2}+it \Big) \Big\rvert^{2k} \mathrm{d}t \sim a_k b_k (\log T)^{k^2 + 2k} ,
\end{align*}
where
\begin{align*}
a_k := \prod_{p \text{ prime}} \Big( 1- \frac{1}{p} \Big)^{k^2} \sum_{m=0}^{\infty} \bigg( \frac{\Gamma (m+k)}{m! \Gamma (k)} \bigg)^2 p^{-m}
\end{align*}
and values for $b_1 , b_2 , \ldots , b_{15}$ are explicitely given. In particular,
\begin{align*}
b_1 = &\frac{1}{3} , \\
b_2 = &\frac{61}{2^5 \cdot 3^2 \cdot 5 \cdot 7} .
\end{align*}
Notice the similarity between these conjectures and the corresponding special cases of our results:
\begin{align*}
\frac{1}{\phi (Q)} \sum_{\substack{\chi \modulus Q \\ \chi \neq \chi_0}} \bigg\lvert L' \bigg( \frac{1}{2} , \chi \bigg) \bigg\rvert^2
\sim (\log q)^{2} \frac{1}{3} (\degree Q)^{3}
\end{align*}
and
\begin{align*}
\frac{1}{\phi (Q)} \sum_{\substack{\chi \modulus Q \\ \chi \neq \chi_0}} \bigg\lvert L' \bigg( \frac{1}{2} , \chi \bigg) \bigg\rvert^4
\sim (\log q)^{4} (1-q^{-1}) \frac{61}{2^5 \cdot 3^2 \cdot 5 \cdot 7} (\degree Q)^{6} .
\end{align*}

Furthermore, we prove the following result:

\begin{theorem} \label{Fourth moment coefficients asymptotic behaviour}
For all non-negative inetgers $m$ we define
\begin{align}
\begin{split} \label{4th moment, mth derivative main coeff limit}
D_m := &\lim_{\degree Q \longrightarrow \infty} \frac{1}{(1-q^{-1}) (\log q)^{4m}} \frac{1}{\phi (Q)} \frac{1}{(\degree Q)^{4m + 4}} \sum_{\substack{\chi \modulus Q \\ \chi \neq \chi_0}} \Big\lvert L^{(m)} \Big( \frac{1}{2} , \chi \Big) \Big\rvert^{4} \\
= &\int_{\substack{a_1 , a_2 , a_3 , a_4 \geq 0 \\ 2a_1 + a_3 + a_4 < 1 \\ 2a_2 + a_3 + a_4 < 1}} \Big( (a_1 + a_3 )^m (a_1 + a_4 )^m + (1- a_1 - a_3 )^m (1 - a_1 - a_4 )^m \Big) \\
&\quad \quad \quad \quad \quad \quad \quad \quad \quad \cdot \Big( (a_2 + a_3 )^m (a_2 + a_4 )^m + (1- a_2 - a_3 )^m (1 - a_2 - a_4 )^m \Big) \mathrm{d} a_1 \mathrm{d} a_2 \mathrm{d} a_3 \mathrm{d} a_4 .
\end{split}
\end{align}
We have that
\begin{align*}
D_m \sim \frac{1}{16 m^4}
\end{align*}
as $m \longrightarrow \infty$.
\end{theorem}

We note the similarity between our result and a result of Conrey's \cite{FourthMoment_Conrey}, which states that
\begin{align*}
\frac{\pi^2}{6} C_{2,m} \sim \frac{1}{16 m^4}
\end{align*}
as $m \longrightarrow \infty$, where
\begin{align*}
C_{2,m}
= \lim_{T \longrightarrow \infty} T^{-1} \bigg( \log \Big(\frac{T}{2 \pi} \Big) \bigg)^{-4m -4} \int_{t=1}^{T} \Big\lvert \zeta^{(m)} \Big( \frac{1}{2} +it \Big) \Big\rvert^{4} \mathrm{d} t .
\end{align*}
Note that the factor of $\zeta (2) = \frac{ \pi^2 }{6}$ in Conrey's result corresponds to the factor of $\frac{1}{\zeta_{\mathcal{A}} (2)} = \frac{1}{1 - q^{-1}}$ in our definition of $D_m$.


\section{Preliminary Results}

The following results are well known and, for many, the proofs can be found in Rosen's book \cite{NumTheoFuncField_Rosen_2002}.

\begin{lemma} \label{Dir Char Sum Over Mult Group}
Let $\chi $ be a non-trivial Dirichlet character modulo $R$. Then, 
\begin{align*} 
\sum_{A \in \mathcal{A}/R\mathcal{A}} \chi (A)
= \sum_{A \in (\mathcal{A}/R\mathcal{A})^*} \chi (A)
= 0 .
\end{align*}
\end{lemma}

\begin{lemma} \label{Sum of chi(A) over base field}
Let $\chi $ be an an odd Dirichlet character. Then,
\begin{align*}
\sum_{a \in {\mathbb{F}_q}} \chi (a)
= \sum_{a \in {\mathbb{F}_q}^*} \chi (a)
= 0 .
\end{align*}
\end{lemma}

\begin{lemma} \label{Sum of chi(A) over all chi modulus R}
Let $R \in \mathcal{M}$. Then,
\begin{align*}
\sum_{\chi \modulus R} \chi (A) 
= \begin{cases}
\phi (R) &\text{ if $A \equiv 1 (\modulus R)$} \\
0 &\text{ otherwise} ,
\end{cases}
\end{align*}
and
\begin{align*}
\sum_{\substack{\chi \modulus R \\ \chi \text{ even}}} \chi (A) 
= \begin{cases}
\frac{\phi (R)}{q-1} &\text{ if $A \equiv a (\modulus R)$ for some $a \in (\mathbb{F}_q)^*$} \\
0 &\text{ otherwise} .
\end{cases} 
\end{align*}
\end{lemma}

\begin{corollary} \label{Sum of chi(A) inv-chi(B) over chi mod R}
Let $R \in \mathcal{M}$. Then,
\begin{align*}
\sum_{\chi \modulus R} \chi (A) \charinv\chi (B)
= \begin{cases}
\phi (R) &\text{ if $(AB,R)=1$ and $A \equiv B (\modulus R)$} \\
0 &\text{ otherwise} ,
\end{cases} 
\end{align*}
and
\begin{align*}
\sum_{\substack{\chi \modulus R \\ \chi \text{ even}}} \chi (A) \charinv\chi (B)
= \begin{cases}
\frac{\phi (R)}{q-1} &\text{ if $(AB,R)=1$ and $A \equiv aB (\modulus R)$ for some $a \in (\mathbb{F}_q)^*$} \\
0 &\text{ otherwise} .
\end{cases} 
\end{align*}
\end{corollary}

\begin{proposition} \label{Analytic continuation and Euler product of function field zeta}
We have that
\begin{align*}
\zeta_{\mathcal{A}} (s)
= \frac{1}{1-q^{1-s}} ,
\end{align*}
and this gives an analytic continuation of $\zeta_{\mathcal{A}}$ to $\mathbb{C} \backslash \{ 1 + \frac{2n \pi i}{\log q} : n \in \mathbb{Z} \}$. We also have the following Euler product for $\Re (s) >1$:
\begin{align*}
\zeta_{\mathcal{A}} (s)
= \prod_{P \in \mathcal{P}} \Big( 1 - \lvert P \rvert^{-s} \Big)^{-1} .
\end{align*}
\end{proposition}

Proposition \ref{Analytic continuation and Euler product of function field zeta} can be generalised to the following:

\begin{proposition}
Let $R \in \mathcal{M}$ and let $\chi$ be a Dirichlet character of modulus $R$. If $\chi = \chi_0$ then we have
\begin{align*}
L (s, \chi_0)
= \Bigg( \prod_{P \mid R} 1+ \lvert P \rvert^{-s} \Bigg) \zeta_{\mathcal{A}} (s) .
\end{align*}
If $\chi \neq \chi_0$ then we have
\begin{align*}
L (s, \chi)
= \sum_{\substack{A \in \mathcal{M} \\ \degree A < \degree R}} \frac{\chi (A)}{\lvert A \rvert^s} .
\end{align*}
We can now see how the analytic continuations given in the introduction are obtained.
\end{proposition}

\begin{definition} \label{Definition, L_n ( chi )}
We will often write
\begin{align*}
L (s, \chi ) = \sum_{n=0}^{\infty} L_n ( \chi ) (q^{-s} )^n,
\end{align*}
where
\begin{align*}
L_n ( \chi ) := \sum_{\substack{A \in \mathcal{M} \\ \degree A = n}} \chi (A) .
\end{align*}
As shown above, if $\chi$ is a non-trivial character of modulus $R$, then $L_n (\chi ) = 0$ for $n \geq \degree R$.
\end{definition}

\begin{lemma} \label{ZetaCoprimeDoubleSum}
For $\Re (r) , \Re (s) > 1$, we have that
\begin{align*}
\sum_{\substack{R,S \in \mathcal{M} \\ (R,S)=1}} \frac{1}{\lvert R \rvert^r \lvert S \rvert^s }
= \Big( \sum_{R \in \mathcal{M}} \frac{1}{\lvert R \rvert^r } \Big) \Big( \sum_{S \in \mathcal{M}} \frac{1}{\lvert S \rvert^s } \Big) \Big( 1 - q^{1-s-r} \Big) .
\end{align*}
\end{lemma}

\begin{proof}
We have that
\begin{align*}
\sum_{R,S \in \mathcal{M}} \frac{1}{\lvert R \rvert^r \lvert S \rvert^s }
= &\sum_{G \in \mathcal{M}} \sum_{\substack{R,S \in \mathcal{M} \\ (R,S) =G}} \frac{1}{\lvert R \rvert^r \lvert S \rvert^s }
= \sum_{G \in \mathcal{M}} \frac{1}{\lvert G \rvert^r \lvert G \rvert^s} \sum_{\substack{R,S \in \mathcal{M} \\ (R,S) =1}} \frac{1}{\lvert R \rvert^r \lvert S \rvert^s } \\
= &\bigg( \sum_{G \in \mathcal{M}} \frac{1}{\lvert G \rvert^{r+s}} \bigg) \bigg( \sum_{\substack{R,S \in \mathcal{M} \\ (R,S) =1}} \frac{1}{\lvert R \rvert^r \lvert S \rvert^s } \bigg) .
\end{align*}
From this we easily deduce that
\begin{align*}
\sum_{\substack{R,S \in \mathcal{M} \\ (R,S) =1}} \frac{1}{\lvert R \rvert^r \lvert S \rvert^s }
= \bigg( \sum_{R,S \in \mathcal{M}} \frac{1}{\lvert R \rvert^r \lvert S \rvert^s } \bigg) \bigg( \sum_{G \in \mathcal{M}} \frac{1}{\lvert G \rvert^{r+s}} \bigg)^{-1}
= \Big( \sum_{R \in \mathcal{M}} \frac{1}{\lvert R \rvert^r } \Big) \Big( \sum_{S \in \mathcal{M}} \frac{1}{\lvert S \rvert^s } \Big) \Big( 1 - q^{1-s-r} \Big) .
\end{align*}
\end{proof}


\section{First Moments} \label {First Moments}

To prove Theorem \ref{first_moment_theorem} we will require the following lemma.

\begin{lemma} \label{Lemma, sum of n^k q^(n/2) over n=0 to deg Q -1}
For all positive integers $k$ we have that
\begin{align*}
\sum_{n=0}^{\degree Q -1} n^k q^{\frac{n}{2}}
= \frac{1}{q^{\frac{1}{2}}-1} (\degree Q)^k \lvert Q \rvert^{\frac{1}{2}} + O_k \bigg( (\degree Q)^{k-1} \lvert Q \rvert^{\frac{1}{2}} \bigg),
\end{align*}
as $\degree Q \longrightarrow \infty$.
\end{lemma}

\begin{proof}
We have that
\begin{align*}
\sum_{n=0}^{\degree Q -1} n^k q^{\frac{n}{2}}
= &\frac{1}{q^{\frac{1}{2}} -1} \sum_{n=0}^{\degree Q -1} \Big( n^k q^{\frac{n+1}{2}} - n^k q^{\frac{n}{2}} \Big) \\
= &\frac{1}{q^{\frac{1}{2}} -1} \sum_{n=0}^{\degree Q -1} \Big( (n+1)^k q^{\frac{n+1}{2}} - n^k q^{\frac{n}{2}} \Big)
- \frac{1}{q^{\frac{1}{2}} -1} \sum_{n=0}^{\degree Q -1} \Big( (n+1)^k q^{\frac{n+1}{2}} - n^k q^{\frac{n+1}{2}} \Big) \\
= &\frac{1}{q^{\frac{1}{2}}-1} (\degree Q)^k \lvert Q \rvert^{\frac{1}{2}}
+O \bigg( \sum_{i=0}^{k-1} \binom{k}{i} (\degree Q)^i  \sum_{n=0}^{\degree Q -1} q^{\frac{n+1}{2}} \bigg) \\
= &\frac{1}{q^{\frac{1}{2}}-1} (\degree Q)^k \lvert Q \rvert^{\frac{1}{2}}
+ O_k \bigg( (\degree Q)^{k-1} \lvert Q \rvert^{\frac{1}{2}} \bigg)
\end{align*}
as $\degree Q \longrightarrow \infty$.
\end{proof}

\begin{proof}[Proof of Theorem \ref{first_moment_theorem}]
We can easily see that
\begin{align*}
L^{(k)} (s,\chi )
= (-\log q)^k \sum_{n=1}^{\degree Q -1} n^k q^{-ns} \sum_{\substack{A \in \mathcal{M} \\ \degree A=n}} \chi (A) ,
\end{align*}
from which we deduce that
\begin{align*}
\frac{1}{\phi (Q)} \sum_{\substack{\chi \modulus Q \\ \chi \neq \chi_0}} L^{(k)} \Big( \frac{1}{2} , \chi  \Big)
= &\frac{(-\log q)^k}{\phi (Q)} \sum_{n=1}^{\degree Q -1} n^k q^{-\frac{n}{2}} \sum_{\substack{A \in \mathcal{M} \\ \degree A=n}} \sum_{\substack{\chi \modulus Q \\ \chi \neq \chi_0}} \chi (A) \\
= &- \frac{(-\log q)^k}{\phi (Q)} \sum_{n=1}^{\degree Q -1} n^k q^{-\frac{n}{2}} \sum_{\substack{A \in \mathcal{M} \\ \degree A=n}} 1 \\
= &-\frac{(-\log q)^k}{q^{\frac{1}{2}}-1} \frac{ (\degree Q)^k}{\lvert Q \rvert^{\frac{1}{2}}}
+ O_k \bigg( (\log q)^k \frac{ (\degree Q)^{k-1}}{\lvert Q \rvert^{\frac{1}{2}}} \bigg)
\end{align*}
as $\degree Q \longrightarrow \infty$. For the second equality we used Lemma \ref{Sum of chi(A) over all chi modulus R}, and for the last equality we used Lemma \ref{Lemma, sum of n^k q^(n/2) over n=0 to deg Q -1} and the fact that $\phi (Q) = \lvert Q \rvert -1$ (since $Q$ is prime).
\end{proof}


\section{Second Moments} \label{Second Moments}

\begin{proof}[Proof of Theorem \ref{second_moment_theorem}]
For positive integers $k$ we have that
\begin{align*}
L^{(k)} \Big( \frac{1}{2},\chi \Big)
= (-\log q)^k \sum_{n=1}^{\degree Q -1} n^k q^{-\frac{n}{2}} \sum_{\substack{A \in \mathcal{M} \\ \degree A=n}} \chi (A)
= (-\log q)^k \sum_{\substack{A \in \mathcal{M} \\ \degree A < \degree Q}} \frac{(\log_q \lvert A \rvert )^k \chi (A)}{\lvert A \rvert^{\frac{1}{2}}} ,
\end{align*}
and so
\begin{align*}
\frac{1}{\phi (Q)} \sum_{\substack{\chi \modulus Q \\ \chi \neq \chi_0}} \bigg\lvert L^{(k)} \bigg( \frac{1}{2} , \chi \bigg) \bigg\rvert^2
= \frac{( \log q )^{2k}}{\phi (Q)} \sum_{\substack{A,B \in \mathcal{M} \\ \degree A , \degree B < \degree Q}} \frac{ ( \log_q \lvert A \rvert \log_q \lvert B \rvert )^k}{\lvert AB \rvert^{\frac{1}{2}}} \sum_{\substack{\chi \modulus Q \\ \chi \neq \chi_0}} \chi (A) \charinv\chi (B) .
\end{align*}
We now apply Corollary \ref{Sum of chi(A) inv-chi(B) over chi mod R} to obtain that
\begin{align*}
&\frac{1}{\phi (Q)} \sum_{\substack{\chi \modulus Q \\ \chi \neq \chi_0}} \bigg\lvert L^{(k)} \bigg( \frac{1}{2} , \chi \bigg) \bigg\rvert^2 \\
= &( \log q )^{2k} \sum_{\substack{A \in \mathcal{M} \\ \degree A < \degree Q}} \frac{ ( \log_q \lvert A \rvert )^{2k}}{\lvert A \rvert}
- \frac{( \log q )^{2k}}{\phi (Q)} \sum_{\substack{A,B \in \mathcal{M} \\ \degree A , \degree B < \degree Q}} \frac{ ( \log_q \lvert A \rvert \log_q \lvert B \rvert )^k}{\lvert AB \rvert^{\frac{1}{2}}} .
\end{align*}

For the first term on the RHS we have that
\begin{align*}
\sum_{\substack{A \in \mathcal{M} \\ \degree A < \degree Q}} \frac{ ( \log_q \lvert A \rvert )^{2k}}{\lvert A \rvert}
= \sum_{n=0}^{\degree Q -1} n^{2k}
= \frac{1}{2k+1} (\degree Q)^{2k+1} + O \Big( (\degree Q)^{2k} \Big) 
\end{align*}
as $\degree Q \longrightarrow \infty$, where the final equality uses Faulhaber's formula. For the second term we have that
\begin{align*}
&\frac{1}{\phi (Q)} \sum_{\substack{A,B \in \mathcal{M} \\ \degree A , \degree B < \degree Q}} \frac{ ( \log_q \lvert A \rvert \log_q \lvert B \rvert )^k}{\lvert AB \rvert^{\frac{1}{2}}}
= \frac{1}{\phi (Q)} \bigg( \sum_{n=0}^{\degree Q -1}  n^k q^{\frac{n}{2}} \bigg)^2 \\
\leq &\frac{1}{\phi (Q)} \bigg( (\degree Q)^k \sum_{n=0}^{\degree Q -1} q^{\frac{n}{2}} \bigg)^2 
\ll \frac{1}{\phi (Q)} \big( (\degree Q)^k \lvert Q \rvert^{\frac{1}{2}} \big)^2 
= O \Big( (\degree Q)^{2k} \Big) 
\end{align*}
as $\degree Q \longrightarrow \infty$. The result now follows.
\end{proof}


\section{Fourth Moments: Expressing as  Manageable Summations} \label{Fourth Moments: Expressing as  Manageable Summations}

Before proceeding to the main part of the proof for the fourth moments, we need to express the fourth moments as more manageable summations. \\

A generalisation of the following theorem appears in Rosen's book \cite[Theorem 9.24 A]{NumTheoFuncField_Rosen_2002}.

\begin{theorem} [Functional Equation for $L (s,\chi )$]
Let $\chi $ be a non-trivial character modulo a monic prime polynomial $Q$. If $\chi $ is an odd character, then $L(s,\chi )$ satisfies the functional equation
\begin{align*}
L(s,\chi ) = W(\chi ) q^{\frac{\degree Q -1}{2}} (q^{-s})^{\degree Q -1} L(1-s, \charinv\chi ) ,
\end{align*}
and if $\chi $ is an even character, then $L(s,\chi )$ satisfies the functional equation
\begin{align*}
(q^{1-s} - 1 ) L(s,\chi ) = W(\chi ) q^{\frac{\degree Q}{2}} (q^{-s} - 1 ) (q^{-s})^{\degree Q -1} L(1-s, \charinv\chi ) ;
\end{align*}
where we always have
\begin{align*}
\lvert W(\chi ) \rvert = 1 .
\end{align*}
\end{theorem}

\begin{proposition} \label{Proposition, odd character L(1/2 ,chi) shorter sum}
Let $\chi$ be an odd character modulo a prime $Q$, and let $k$ be a non-negative integer. Then,
\begin{align*}
&(\log q)^{-2k} \Big\lvert L^{(k)} \Big( \frac{1}{2}, \chi \Big) \Big\rvert^2 \\
= &\sum_{\substack{A,B \in \mathcal{M} \\ \degree AB < \degree Q}} \frac{ \Big( f_{k} \big( \degree A, \degree B, \degree Q \big) + g_{O,k} \big( \degree A, \degree B, \degree Q \big) \Big) \chi (A) \charinv\chi (B)}{\lvert AB \rvert^{\frac{1}{2}}} \\
&+ \sum_{\substack{A,B \in \mathcal{M} \\ \degree AB = \degree Q -1}} \frac{ h_{O,k} \big( \degree A, \degree B, \degree Q \big) \chi (A) \charinv\chi (B)}{\lvert AB \rvert^{\frac{1}{2}}} ,
\end{align*}
where
\begin{align*}
f_{k} \big( \degree A, \degree B, \degree Q \big) = &(\degree A )^k (\degree B)^k + (\degree Q -\degree A)^k (\degree Q -\degree B)^k , \\
g_{O,k} \big( \degree A, \degree B, \degree Q \big) = &(\degree Q -\degree A -1)^k (\degree Q -\degree B -1)^k \\
&- (\degree Q -\degree A)^k (\degree Q -\degree B)^k , \\
h_{O,k} \big( \degree A, \degree B, \degree Q \big) = &- (\degree Q -\log_q \lvert A \rvert -1)^k (\degree Q -\log_q \lvert B \rvert -1)^k .
\end{align*}
\end{proposition}

\begin{remark}
The ``$O$" in the subscript is to signify that these polynomials apply to the odd character case. It is important to note that $g_{O,k} \big( \degree A, \degree B, \degree Q \big)$ has degree $2k-1$, whereas \\
$f_{k} \big( \degree A, \degree B, \degree Q \big)$ has degree $2k$ (hence, the later ultimately conributes the higher order term); and that all three polynomials are independent of $q$.
\end{remark}

\begin{proof}
The functional equation gives us that
\begin{align*}
 \sum_{n=0}^{\degree Q-1} L_n (\chi) (q^{-s})^n
= &W (\chi) q^{\frac{\degree Q-1}{2}} (q^{-s})^{\degree Q-1} \sum_{n=0}^{\degree Q-1}  L_n (\charinv\chi) (q^{s-1})^n \\
= &W (\chi) q^{-\frac{\degree Q-1}{2}} \sum_{n=0}^{\degree Q-1} L_n (\charinv\chi) (q^{1-s})^{\degree Q-n-1} .
\end{align*}
Taking the $k^{\text{th}}$ derivative of both sides gives
\begin{align*}
&(-\log q)^k  \sum_{n=0}^{\degree Q-1} n^k L_n (\chi) (q^{-s})^n \\
= &(-\log q)^k W (\chi) q^{-\frac{\degree Q-1}{2}} \sum_{n=0}^{\degree Q-1} (\degree Q-n-1)^k  L_n (\charinv\chi) (q^{1-s})^{\degree Q-n-1} .
\end{align*}
Let us now take the squared modulus of both sides, to get
\begin{align*}
&(\log q)^{2k}  \sum_{n=0}^{2 \degree Q-2} \bigg( \sum_{\substack{i+j=n \\ 0 \leq i,j < \degree Q}} i^k j^k L_i (\chi) L_j (\charinv\chi)\bigg) (q^{-s})^n \\
= &(\log q)^{2k} q^{-\degree Q +1} \sum_{n=0}^{2 \degree Q-2} \bigg( \sum_{\substack{i+j=n \\ 0 \leq i,j < \degree Q}} (\degree Q -i-1)^k (\degree Q -j-1)^k L_i (\chi) L_j (\charinv\chi)\bigg) (q^{1-s})^{2\degree Q -n-2} .
\end{align*}

Both sides of the above are equal to $\big\lvert L^{(k)} (s, \chi) \big\rvert^2$. By the linear independence of powers of $q^{-s}$, we have that $\big\lvert L^{(k)} (s, \chi) \big\rvert^2$ is the sum of the terms corresponding to $n=0, 1, \ldots , \degree Q -1$ from the LHS and $n=0, 1, \ldots , \degree Q -2$ from the RHS. This gives
\begin{align*}
&(\log q)^{-2k} \big\lvert L^{(k)} (s, \chi) \big\rvert^2 \\
= &\sum_{n=0}^{\degree Q-1} \bigg( \sum_{\substack{i+j=n \\ 0 \leq i,j < \degree Q}} i^k j^k L_i (\chi) L_j (\charinv\chi)\bigg) (q^{-s})^n \\
&+ q^{-\degree Q +1} \sum_{n=0}^{\degree Q-2} \bigg( \sum_{\substack{i+j=n \\ 0 \leq i,j < \degree Q}} (\degree Q -i-1)^k (\degree Q -j-1)^k L_i (\chi) L_j (\charinv\chi)\bigg) (q^{1-s})^{2\degree Q -n-2} .
\end{align*}
We now substitue $s=\frac{1}{2}$ and simplify the right-hand-side to get
\begin{align*}
&(\log q)^{-2k} \Big\lvert L^{(k)} \Big( \frac{1}{2}, \chi \Big) \Big\rvert^2 \\
= &\sum_{n=0}^{\degree Q-1} \bigg( \sum_{\substack{i+j=n \\ 0 \leq i,j < \degree Q}} i^k j^k L_i (\chi) L_j (\charinv\chi)\bigg) q^{-\frac{n}{2}} \\
&+ \sum_{n=0}^{\degree Q-2} \bigg( \sum_{\substack{i+j=n \\ 0 \leq i,j < \degree Q}} (\degree Q -i-1)^k (\degree Q -j-1)^k L_i (\chi) L_j (\charinv\chi)\bigg) q^{-\frac{n}{2}} \\
= &\sum_{n=0}^{\degree Q-1} \bigg( \sum_{\substack{i+j=n \\ 0 \leq i,j < \degree Q}} \Big[ i^k j^k + (\degree Q -i-1)^k (\degree Q -j-1)^k \Big] L_i (\chi) L_j (\charinv\chi)\bigg) q^{-\frac{n}{2}} \\
&- \sum_{\substack{i+j=\degree Q-1 \\ 0 \leq i,j < \degree Q}} (\degree Q -i-1)^k (\degree Q -j-1)^k L_i (\chi) L_j (\charinv\chi) q^{-\frac{\degree Q-1}{2}} .
\end{align*}
Finally, we substitute back $L_n (\chi) = \sum_{\substack{A \in \mathcal{M} \\ \degree A = n}} \chi (A)$ to obtain the required result.
\end{proof}

\begin{definition} \label{L hat definition}
For all $s \in \mathbb{C}$ and all non-trivial even characters, $\chi$, of prime modulus we define
\begin{align} \label{L hat definition, equation}
\hat{L} (s, \chi )
:= (q^{1-s} -1) L(s, \chi ) .
\end{align}
\end{definition}

\begin{proposition} \label{Proposition, L in terms of L hat}
For all non-trivial even characters, $\chi$, of prime modulus and all non-negative integers $k$ we have that
\begin{align*}
L^{(k)} \Big( \frac{1}{2} , \chi \Big)
= &\frac{1}{q^{\frac{1}{2}} -1} \hat{L}^{(k)} \Big( \frac{1}{2} , \chi \Big)
+ \frac{1}{q^{\frac{1}{2}} -1} \sum_{i=0}^{k-1} (-\log q)^{k-i} p_{k,i} \Big( \frac{q^{\frac{1}{2}}}{q^{\frac{1}{2}}-1} \Big)  \hat{L}^{(i)} \Big( \frac{1}{2} , \chi \Big) \\
= &\frac{1}{q^{\frac{1}{2}} -1} \sum_{i=0}^{k} (-\log q)^{k-i} p_{k,i} \Big( \frac{q^{\frac{1}{2}}}{q^{\frac{1}{2}}-1} \Big)  \hat{L}^{(i)} \Big( \frac{1}{2} , \chi \Big) ,
\end{align*}
where, for non-negative integers $k,i$ satisfying $i \leq k$, we define the polynomials $p_{k,i}$ by
\begin{align*}
p_{k,k} \Big( \frac{q^{\frac{1}{2}}}{q^{\frac{1}{2}}-1} \Big) = &1 , \\
p_{k,i} \Big( \frac{q^{\frac{1}{2}}}{q^{\frac{1}{2}}-1} \Big) = &- \frac{q^{\frac{1}{2}}}{q^{\frac{1}{2}} -1} \sum_{j=i}^{k-1} \binom{k}{j} p_{j,i} \Big( \frac{q^{\frac{1}{2}}}{q^{\frac{1}{2}}-1} \Big)
\quad \quad \text{ for $i<k$.}
\end{align*}
\end{proposition}

\begin{remark}
Because $\frac{q^{\frac{1}{2}}}{q^{\frac{1}{2}}-1} < 4$ for all prime powers $q$, we can see that the polynomials $p_{k,i} \Big( \frac{q^{\frac{1}{2}}}{q^{\frac{1}{2}}-1} \Big)$ can be bounded independently of $q$ (but dependent on $k$ and $i$ of course). The factors $(-\log q)^{k-i}$ are of course still dependent on $q$, as well as $k$ and $i$. These two points are imporant when we later determine how the lower order terms in our main results are dependent on $q$.
\end{remark}

\begin{proof}
We prove this by strong induction on $k$. The base case, $k=0$, is obvious by Definition \ref{L hat definition}. Now, suppose the claim holds for $j=0,1,\ldots , k$. Differentiating, $k+1$ number of times, the equation (\ref{L hat definition, equation}) gives
\begin{align*}
\hat{L}^{(k+1)} (s, \chi )
= (q^{1-s} -1) L^{(k+1)} (s, \chi ) + q^{1-s} \sum_{j=0}^{k} \binom{k+1}{j} (- \log q)^{k+1-j} L^{(j)} (s, \chi ) .
\end{align*}
Substituting $s=\frac{1}{2}$ and rearranging gives
\begin{align*}
L^{(k+1)} \Big( \frac{1}{2}, \chi \Big)
= \frac{1}{q^{\frac{1}{2}} -1} \hat{L}^{(k+1)} \Big( \frac{1}{2}, \chi \Big) - \frac{q^{\frac{1}{2}}}{q^{\frac{1}{2}} -1} \sum_{j=0}^{k} \binom{k+1}{j} (- \log q)^{k+1-j} L^{(j)} \Big( \frac{1}{2}, \chi \Big) .
\end{align*}
We now apply the inductive hypothesis to obatin
\begin{align*}
&L^{(k+1)} \Big( \frac{1}{2}, \chi \Big) \\
= &\frac{1}{q^{\frac{1}{2}} -1} \hat{L}^{(k+1)} \Big( \frac{1}{2}, \chi \Big) \\
&- \frac{q^{\frac{1}{2}}}{q^{\frac{1}{2}} -1} \sum_{j=0}^{k} \binom{k+1}{j} (- \log q)^{k+1-j} \frac{1}{q^{\frac{1}{2}} -1} \sum_{i=0}^{j} (-\log q)^{j-i} p_{j,i} \Big( \frac{q^{\frac{1}{2}}}{q^{\frac{1}{2}}-1} \Big)  \hat{L}^{(i)} \Big( \frac{1}{2} , \chi \Big) \\
= &\frac{1}{q^{\frac{1}{2}} -1} \hat{L}^{(k+1)} \Big( \frac{1}{2}, \chi \Big) \\
&+ \frac{1}{q^{\frac{1}{2}} -1} \sum_{i=0}^{k} (- \log q)^{k+1-i} \bigg( - \frac{q^{\frac{1}{2}}}{q^{\frac{1}{2}} -1} \sum_{j=i}^{k} \binom{k+1}{j} p_{j,i} \Big( \frac{q^{\frac{1}{2}}}{q^{\frac{1}{2}}-1} \Big) \bigg)  \hat{L}^{(i)} \Big( \frac{1}{2} , \chi \Big) .
\end{align*}
The result folows by the definition of the polynomials $p_{k,i}$ .
\end{proof}

\begin{proposition} \label{Proposition, even character L(1/2 ,chi) shorter sum}
For all non-negative integers $k$, and all non-trivial even characters $\chi$ of prime modulus $Q$, we have that
\begin{align*}
&\frac{1}{(\log q)^{2k} (q^{\frac{1}{2}} -1)^2} \Big\lvert \hat{L}^{k} \Big( \frac{1}{2}, \chi \Big) \Big\lvert^2 \\
= &\sum_{\substack{A,B \in \mathcal{M} \\ \degree AB < \degree Q}} \frac{f_{k} \big( \degree A, \degree B, \degree Q \big) + g_{E , k} \big( \degree A, \degree B, \degree Q \big)}{\lvert AB \rvert^{\frac{1}{2}}} \\
&+ \sum_{\degree Q -2 \leq n \leq \degree Q}  \sum_{\substack{A,B \in \mathcal{M} \\ \degree AB = n}} \frac{h_{E , k , n} \big( \degree A, \degree B, \degree Q \big) }{\lvert AB \rvert^{\frac{1}{2}}} ,
\end{align*}
where
\begin{align*}
f_{k} \big( \degree A, \degree B, \degree Q \big)
= (\degree A)^{k} (\degree B)^{k} + (\degree Q - \degree A)^{k} (\degree Q - \degree B)^{k} ,
\end{align*}
 and $g_{E , k} \big( \degree A, \degree B, \degree Q \big) \; , \; h_{E , k, n} \big( \degree A, \degree B, \degree Q \big)$ are polynomials of degrees $2k -1$ and $2k$, respectively, whose coefficients can be bounded independently of $q$.
\end{proposition}

\begin{proof}
Let us define $L_{-1} (\chi ) := 0$, and recall from Definition \ref{Definition, L_n ( chi )} that $L_{\degree Q} (\chi) = 0$. We can now define, for $n=0,1,\ldots , \degree Q$,
\begin{align*}
M_n (\chi) := L_n (\chi ) - q L_{n-1} (\chi ) .
\end{align*}
Then, the functional equation for even characters can be written as 
\begin{align} \label{Even func eq with M_n}
- \sum_{n=0}^{\degree Q} M_n (\chi ) (q^{-s})^n
= W(\chi ) q^{-\frac{\degree Q}{2}} \sum_{n=0}^{\degree Q} M_n (\charinv\chi ) (q^{1-s})^{\degree Q -n} .
\end{align}
Note that both sides of (\ref{Even func eq with M_n}) are equal to $\hat{L} (s, \chi )$. We proceed similar to the odd character case. First we differentiate, $k$ number of times, the equation (\ref{Even func eq with M_n}); and then we take the modulus squared of both sides. This gives
\begin{align*}
&(\log q)^{2k} \sum_{n=0}^{2 \degree Q} \bigg( \sum_{\substack{i+j=n \\ 0 \leq i,j \leq \degree Q}} i^k j^k M_i (\chi ) M_j (\charinv\chi ) \bigg) (q^{-s})^n \\
= &(\log q)^{2k} q^{- \degree Q} \sum_{n=0}^{2 \degree Q} \bigg( \sum_{\substack{i+j=n \\ 0 \leq i,j \leq \degree Q}}  (\degree Q -i)^k (\degree Q -j)^k M_i (\chi ) M_j (\charinv\chi ) \bigg) (q^{1-s})^{2\degree Q -n} .
\end{align*}
Now we take the terms corresponding to $n=0,1, \ldots , \degree Q$ from the LHS and $n=0,1,\ldots ,\degree Q -1$ from the RHS to obtain
\begin{align*}
\hat{L}^{(k)} (s, \chi ) 
= &(\log q)^{2k} \sum_{n=0}^{\degree Q} \bigg( \sum_{\substack{i+j=n \\ 0 \leq i,j \leq \degree Q}} i^k j^k M_i (\chi ) M_j (\charinv\chi ) \bigg) (q^{-s})^n \\
+ &(\log q)^{2k} q^{- \degree Q} \sum_{n=0}^{\degree Q -1} \bigg( \sum_{\substack{i+j=n \\ 0 \leq i,j \leq \degree Q}}  (\degree Q -i)^k (\degree Q -j)^k M_i (\chi ) M_j (\charinv\chi ) \bigg) (q^{1-s})^{2\degree Q -n} .
\end{align*}
Substituting $s=\frac{1}{2}$ and simplifying the RHS gives
\begin{align}
\begin{split} \label{Even func eq after s=1/2 and simplifying}
\hat{L}^{(k)} \Big( \frac{1}{2}, \chi \Big)
= &(\log q)^{2k} \sum_{n=0}^{\degree Q -1} \bigg( \sum_{\substack{i+j=n \\ 0 \leq i,j \leq \degree Q}} \Big( i^k j^k + (\degree Q -i)^k (\degree Q -j)^k \Big) M_i (\chi ) M_j (\charinv\chi ) \bigg) q^{-\frac{n}{2}} \\
&+ (\log q)^{2k} \sum_{\substack{i+j= \degree Q \\ 0 \leq i,j \leq \degree Q}} i^k j^k M_i (\chi ) M_j (\charinv\chi ) q^{-\frac{\degree Q}{2}} .
\end{split}
\end{align}

Now, we want factors such as $L_n (\chi)$ in our expression, as opposed to factors like $M_n (\chi )$. To this end, suppose $p (i, j)$ is a finite polynomial. Then,
\begin{align*}
&\sum_{n=0}^{\degree Q -1} \bigg( \sum_{\substack{i+j=n \\ 0 \leq i,j \leq \degree Q}} p (i,j) M_i (\chi ) M_j (\charinv\chi ) \bigg) q^{-\frac{n}{2}} \\
= & \sum_{n=0}^{\degree Q -1} \bigg( \sum_{\substack{i+j=n \\ 0 \leq i,j \leq \degree Q}} p (i,j) \Big( L_i (\chi ) - q L_{i-1} (\chi ) \Big) \Big( L_j (\charinv\chi ) - q L_{j-1} (\charinv\chi ) \Big) \bigg) q^{-\frac{n}{2}} \\
= & \sum_{n=0}^{\degree Q -1} \bigg( \sum_{\substack{i+j=n \\ 0 \leq i,j \leq \degree Q}} p (i,j) L_i (\chi ) L_j (\charinv\chi ) \bigg) q^{-\frac{n}{2}} 
+ \sum_{n=0}^{\degree Q -3} \bigg( \sum_{\substack{i+j=n \\ 0 \leq i,j \leq \degree Q}} p (i+1,j+1) L_i (\chi ) L_j (\charinv\chi ) \bigg) q^{-\frac{n-2}{2}} \\
- &\sum_{n=0}^{\degree Q -2} \bigg( \sum_{\substack{i+j=n \\ 0 \leq i,j \leq \degree Q}} p (i,j+1) L_i (\chi ) L_j (\charinv\chi ) \bigg) q^{-\frac{n-1}{2}} 
- \sum_{n=0}^{\degree Q -2} \bigg( \sum_{\substack{i+j=n \\ 0 \leq i,j \leq \degree Q}} p (i+1,j) L_i (\chi ) L_j (\charinv\chi ) \bigg) q^{-\frac{n-1}{2}} .
\end{align*}
Grouping the terms together gives
\begin{align*}
&\sum_{n=0}^{\degree Q -1} \bigg( \sum_{\substack{i+j=n \\ 0 \leq i,j \leq \degree Q}} p (i,j) M_i (\chi ) M_j (\charinv\chi ) \bigg) q^{-\frac{n}{2}} \\
= &\sum_{n=0}^{\degree Q -1} \bigg( \sum_{\substack{i+j=n \\ 0 \leq i,j \leq \degree Q}} \Big[ q p(i+1,j+1) - q^{\frac{1}{2}} p (i,j+1) - q^{\frac{1}{2}} p (i+1,j) + p (i,j) \Big] L_i (\chi ) L_j (\charinv\chi ) \bigg) q^{-\frac{n}{2}} \\
- &\sum_{\substack{i+j=\degree Q -2 \\ 0 \leq i,j \leq \degree Q}} q p(i+1,j+1)  L_i (\chi ) L_j (\charinv\chi ) q^{-\frac{\degree Q -2 }{2}} \\
+ &\sum_{\substack{i+j=\degree Q -1 \\ 0 \leq i,j \leq \degree Q}}  \Big( q^{\frac{1}{2}} p (i,j+1) + q^{\frac{1}{2}} p (i+1,j) - q p(i+1,j+1) \Big)  L_i (\chi ) L_j (\charinv\chi ) q^{\frac{\degree Q -1 }{2}} .
\end{align*}
In the case where 
\begin{align*}
p(i,j)
= i^k j^k + (\degree Q -i)^k (\degree Q -j)^k
\end{align*}
we have that
\begin{align*}
q p(i+1,j+1) - q^{\frac{1}{2}} p (i,j+1) - q^{\frac{1}{2}} p (i+1,j) + p (i,j) 
= (q^{\frac{1}{2}} -1)^2 \Big( f_k (i,j, \degree Q) + g_{E,k} (i,j, \degree Q) \Big) ,
\end{align*}
where $g_{E,k} (i,j, \degree Q)$ is a polynomial of degree $2k-1$ whose coefficients can be bounded independently of $q$. \\

We can now see that (\ref{Even func eq after s=1/2 and simplifying}) becomes
\begin{align*}
\frac{1}{(\log q)^{2k} (q^{\frac{1}{2}} -1)^2} \hat{L}^{(k)} \Big( \frac{1}{2}, \chi \Big)
= & \sum_{n=0}^{\degree Q -1} \bigg( \sum_{\substack{i+j=n \\ 0 \leq i,j \leq \degree Q}} \Big( f_k (i,j, \degree Q) + g_{E,k} (i,j, \degree Q) \Big) L_i (\chi ) L_j (\charinv\chi ) \bigg) q^{-\frac{n}{2}} \\
&+ \sum_{n= \degree Q -2}^{\degree Q} \bigg( \sum_{\substack{i+j= n \\ 0 \leq i,j \leq \degree Q}} h_{E,k,n} (i,j, \degree Q) L_i (\chi ) L_j (\charinv\chi ) \bigg) q^{-\frac{n}{2}} ,
\end{align*}
where $h_{E,k,n} (i,j, \degree Q)$ is a polynomial of degree $k$ whose coefficients can be bounded independently of $q$. Finally, we substitute back $L_n (\chi) = \sum_{\substack{A \in \mathcal{M} \\ \degree A = n}} \chi (A)$ to obtain the required result.
\end{proof}


\section{Fourth Moments: Handling the Summations}

We now demonstrate some techniques for handling the summations that we obtained in Section \ref{Fourth Moments: Expressing as  Manageable Summations}.

\begin{lemma} \label{Lemma, all nontriv char split to diag, off diag, remainder}
Let $Q \in \mathcal{M}$ be prime, and let $p_1 \big( \degree A, \degree B, \degree Q \big)$ and $p_2 \big( \degree A, \degree B, \degree Q \big)$ be finite polynomials (which, for presentational purposes, we will write as $p_1$ and $p_2$, except when we need to use other variables for the parametres). Then,
\begin{align*}
&\frac{1}{\phi (Q)} \sum_{\substack{\chi \modulus Q \\ \chi \neq \chi_0}} \bigg( \sum_{\substack{ A,B \in \mathcal{M} \\ \degree AB < \degree Q}} \frac{p_1 \; \chi (A) \charinv\chi (B)}{\lvert AB \rvert^{\frac{1}{2}}} \bigg) \bigg( \sum_{\substack{ C,D \in \mathcal{M} \\ \degree CD < \degree Q}} \frac{p_2 \; \chi (C) \charinv\chi (D)}{\lvert CD \rvert^{\frac{1}{2}}} \bigg) \\
= &\sum_{\substack{A,B,C,D \in \mathcal{M} \\ \degree AB < \degree Q \\ \degree CD < \degree Q  \\ AC = BD}} \frac{p_1 p_2}{\lvert ABCD \rvert^{\frac{1}{2}}}
+ \sum_{\substack{A,B,C,D \in \mathcal{M} \\ \degree AB < \degree Q \\ \degree CD < \degree Q \\ AC \equiv BD (\modulus Q)  \\ AC \neq BD}} \frac{p_1 p_2}{\lvert ABCD \rvert^{\frac{1}{2}}}
- \frac{1}{\phi (Q)} \sum_{\substack{A,B,C,D \in \mathcal{M} \\ \degree AB < \degree Q \\ \degree CD < \degree Q}} \frac{p_1 p_2}{\lvert ABCD \rvert^{\frac{1}{2}}} .
\end{align*}
\end{lemma}

\begin{proof}
This follows by expanding the brackets and applying Corollary \ref{Sum of chi(A) inv-chi(B) over chi mod R}
\end{proof}

\begin{lemma}
Let $p \big( \degree A, \degree B, \degree C , \degree D, \degree Q \big)$ be a finite homogeneous polynomial of degree $d$. Then,
\begin{align*}
&\sum_{\substack{A,B,C,D \in \mathcal{M} \\ \degree AB < \degree Q \\ \degree CD < \degree Q  \\ AC = BD}} \frac{p \big( \degree A, \degree B, \degree C , \degree D, \degree Q \big)}{\lvert ABCD \rvert^{\frac{1}{2}}} \\
= &(1- q^{-1}) (\degree Q)^{d+4} \int_{\substack{a_1 , a_2 , a_3 , a_4 \geq 0 \\ 2a_1 + a_3 + a_4 < 1 \\ 2a_2 + a_3 + a_4 < 1}}  p \big( a_1 + a_3 , a_1 + a_4 , a_2 + a_4 , a_2 + a_3 , 1 \big) \mathrm{d} a_1 \mathrm{d} a_2 \mathrm{d} a_3 \mathrm{d} a_4 \\
& + O_p \big( (\degree Q)^{d+3} \big) ,
\end{align*}
as $\degree Q \longrightarrow \infty$.
\end{lemma}

\begin{remark}
The subscript $p$ in $O_p$ should be interpreted as saying that the implied constant is dependent on the coefficients of $p$, but not dependent on the degree $d$.
\end{remark}

\begin{proof}
Consider the function $f$ defined by
\begin{align} \label{f(t_1 ... t_4) with degree A ... D}
f(t_1 , t_2 , t_3 , t_4 )
= \sum_{\substack{A,B,C,D \in \mathcal{M} \\ AC=BD}} \frac{{t_1}^{\degree A} {t_2}^{\degree B} {t_3}^{\degree C} {t_4}^{\degree D}}{\lvert ABCD \rvert^{\frac{1}{2}}}
\end{align}
with domain $\lvert t_i \rvert < \frac{1}{2} q^{-\frac{1}{2}}$. Note that $AC=BD$ if and only if there exist $G,H,R,S \in \mathcal{M}$ satisfying $(R,S)=1$ and $A=GR$, $B=GS$, $C=HS$, $D=HR$. Hence,
\begin{align}
\begin{split} \label{f(t_1 ... t_4) with degree FR ... GR}
&f(t_1 , t_2 , t_3 , t_4 ) \\
&= \sum_{\substack{G,H,R,S \in \mathcal{M} \\ (R,S)=1}} \frac{{t_1}^{\degree GR} {t_2}^{\degree GS} {t_3}^{\degree HS} {t_4}^{\degree HR}}{\lvert GHRS \rvert} \\
&= \sum_{G,H,R,S \in \mathcal{M}} \frac{{t_1}^{\degree GR} {t_2}^{\degree GS} {t_3}^{\degree HS} {t_4}^{\degree HR}}{\lvert GHRS \rvert}
- q^{-1} \sum_{G,H,R,S \in \mathcal{M}} \frac{{t_1}^{\degree GR +1} {t_2}^{\degree GS +1} {t_3}^{\degree HS +1} {t_4}^{\degree HR +1}}{\lvert GHRS \rvert} \\
&= \sum_{ a_1 , a_2 , a_3 , a_4 \geq 0} {t_1}^{a_1 + a_3} {t_2}^{a_1 + a_4} {t_3}^{a_2 + a_4} {t_4}^{a_2 + a_3}
- q^{-1} \sum_{ a_1 , a_2 , a_3 , a_4 \geq 0} {t_1}^{a_1 + a_3 +1} {t_2}^{a_1 + a_4 +1} {t_3}^{a_2 + a_4 +1} {t_4}^{a_2 + a_3 +1} ,
\end{split}
\end{align}
where the second equality follows by similar means as in the proof of Lemma \ref{ZetaCoprimeDoubleSum}. \\

Now, for $i=1,2,3,4$ we define the operator $\Omega_i := t_i \frac{\mathrm{d}}{\mathrm{d} t_i}$. For non-negative integers $k_1 , k_2 , k_3 , k_4$ we can apply the operator ${\Omega_1 }^{k_1} {\Omega_2 }^{k_2} {\Omega_3 }^{k_3} {\Omega_4 }^{k_4}$ to (\ref{f(t_1 ... t_4) with degree A ... D}) and (\ref{f(t_1 ... t_4) with degree FR ... GR}) to get
\begin{align*}
&\sum_{\substack{A,B,C,D \in \mathcal{M} \\ AC=BD}} \frac{(\degree A)^{k_1} (\degree B)^{k_2} (\degree C)^{k_3} (\degree D)^{k_4}}{\lvert ABCD \rvert^{\frac{1}{2}}} {t_1}^{\degree A} {t_2}^{\degree B} {t_3}^{\degree C} {t_4}^{\degree D} \\
&= \sum_{ a_1 , a_2 , a_3 , a_4 \geq 0} (a_1 + a_3 )^{k_1} (a_1 + a_4 )^{k_2} (a_2 + a_4 )^{k_3} (a_2 + a_3 )^{k_4} {t_1}^{a_1 + a_3} {t_2}^{a_1 + a_4} {t_3}^{a_2 + a_4} {t_4}^{a_2 + a_3} \\
& \quad - q^{-1} \sum_{ a_1 , a_2 , a_3 , a_4 \geq 0} (a_1 + a_3 +1)^{k_1} (a_1 + a_4 +1)^{k_2} (a_2 + a_4 +1)^{k_3} (a_2 + a_3 +1)^{k_4} \\
& \quad \quad \quad \quad \quad \quad \quad \quad \quad \quad \quad \quad \cdot {t_1}^{a_1 + a_3 +1} {t_2}^{a_1 + a_4 +1} {t_3}^{a_2 + a_4 +1} {t_4}^{a_2 + a_3 +1} \\
&= (1 - q^{-1} ) \sum_{ a_1 , a_2 , a_3 , a_4 \geq 0} (a_1 + a_3 )^{k_1} (a_1 + a_4 )^{k_2} (a_2 + a_4 )^{k_3} (a_2 + a_3 )^{k_4} {t_1}^{a_1 + a_3} {t_2}^{a_1 + a_4} {t_3}^{a_2 + a_4} {t_4}^{a_2 + a_3} \\
& \quad + q^{-1} \sum_{\substack{(a_1 , a_2 ) = (0,0) , (0,1) , (1,0) \\ a_3 , a_4 \geq 0}} (a_1 + a_3 )^{k_1} (a_1 + a_4 )^{k_2} (a_2 + a_4 )^{k_3} (a_2 + a_3 )^{k_4} {t_1}^{a_1 + a_3 } {t_2}^{a_1 + a_4 } {t_3}^{a_2 + a_4 } {t_4}^{a_2 + a_3 } .
\end{align*}
From this we can deduce that if $p \big( \degree A, \degree B, \degree C , \degree D, \degree Q \big)$ is a finite homogeneous polyomial of degree $d$, then
\begin{align*}
&\sum_{\substack{A,B,C,D \in \mathcal{M} \\ AC=BD}} \frac{p \big( \degree A, \degree B, \degree C , \degree D, \degree Q \big)}{\lvert ABCD \rvert^{\frac{1}{2}}} {t_1}^{\degree A} {t_2}^{\degree B} {t_3}^{\degree C} {t_4}^{\degree D} \\
&= (1- q^{-1}) \sum_{ a_1 , a_2 , a_3 , a_4 \geq 0}  p \big( a_1 + a_3 , a_1 + a_4 , a_2 + a_4 , a_2 + a_3 , \degree Q \big) {t_1}^{a_1 + a_3} {t_2}^{a_1 + a_4} {t_3}^{a_2 + a_4} {t_4}^{a_2 + a_3} \\
& + q^{-1} \sum_{\substack{(a_1 , a_2 ) = (0,0) , (0,1) , (1,0) \\ a_3 , a_4 \geq 0}}  p \big( a_1 + a_3 , a_1 + a_4 , a_2 + a_4 , a_2 + a_3 , \degree Q \big) {t_1}^{a_1 + a_3 } {t_2}^{a_1 + a_4 } {t_3}^{a_2 + a_4 } {t_4}^{a_2 + a_3 } .
\end{align*}

Now, we can extract and sum the coefficients of ${t_1}^{i_1} {t_2}^{i_2} {t_3}^{i_3} {t_4}^{i_4}$ for which $i_1 + i_2 < \degree Q$ and $i_3 + i_4 < \degree Q$ to get
\begin{align*}
&\sum_{\substack{A,B,C,D \in \mathcal{M} \\ \degree AB < \degree Q \\ \degree CD < \degree Q \\ AC=BD}} \frac{p \big( \degree A, \degree B, \degree C , \degree D, \degree Q \big)}{\lvert ABCD \rvert^{\frac{1}{2}}} \\
&= (1- q^{-1}) \sum_{\substack{ a_1 , a_2 , a_3 , a_4 \geq 0 \\ 2a_1 + a_3 + a_4 < \degree Q \\ 2a_2 + a_3 + a_4 < \degree Q}}  p \big( a_1 + a_3 , a_1 + a_4 , a_2 + a_4 , a_2 + a_3 , \degree Q \big) \\
& \quad + q^{-1} \sum_{\substack{(a_1 , a_2 ) = (0,0) , (0,1) , (1,0) \\ a_3 , a_4 \geq 0 \\ 2a_1 + a_3 + a_4 < \degree Q \\ 2a_2 + a_3 + a_4 < \degree Q}}  p \big( a_1 + a_3 , a_1 + a_4 , a_2 + a_4 , a_2 + a_3 , \degree Q \big) \\
&= (1- q^{-1}) \int_{\substack{a_1 , a_2 , a_3 , a_4 \geq 0 \\ 2a_1 + a_3 + a_4 < \degree Q \\ 2a_2 + a_3 + a_4 < \degree Q}}  p \big( a_1 + a_3 , a_1 + a_4 , a_2 + a_4 , a_2 + a_3 , \degree Q \big) \mathrm{d} a_1 \mathrm{d} a_2 \mathrm{d} a_3 \mathrm{d} a_4 \\
& \quad + O_p \big( (\degree Q)^{d+3} \big) + O_p \big( (\degree Q)^{d+2} \big) \\
&= (1- q^{-1}) (\degree Q)^{d+4} \int_{\substack{a_1 , a_2 , a_3 , a_4 \geq 0 \\ 2a_1 + a_3 + a_4 < 1 \\ 2a_2 + a_3 + a_4 < 1}}  p \big( a_1 + a_3 , a_1 + a_4 , a_2 + a_4 , a_2 + a_3 , 1 \big) \mathrm{d} a_1 \mathrm{d} a_2 \mathrm{d} a_3 \mathrm{d} a_4 \\
& \quad + O_p \big( (\degree Q)^{d+3} \big) 
\end{align*}
as $\degree Q \longrightarrow \infty$.
\end{proof}

\begin{lemma}
Let $p \big( \degree A, \degree B, \degree C , \degree D, \degree Q \big)$ be a finite polynomial of degree $d$. Then,
\begin{align*}
\sum_{\substack{A,B,C,D \in \mathcal{M} \\ \degree AB < \degree Q \\ \degree CD < \degree Q \\ AC \equiv BD (\modulus Q)  \\ AC \neq BD}} \frac{p \big( \degree A, \degree B, \degree C , \degree D, \degree Q \big)}{\lvert ABCD \rvert^{\frac{1}{2}}}
\ll_p (\degree Q)^{d+3}
\end{align*}
as $\degree Q \longrightarrow \infty$.
\end{lemma}

\begin{proof}
Because $\degree AB , \degree CD < \degree Q$, we have that
\begin{align*}
p \big( \degree A, \degree B, \degree C , \degree D, \degree Q \big)
\ll_p (\degree Q)^{d} .
\end{align*}
Hence,
\begin{align}
\begin{split} \label{Off-diagonal z_1, z_2 split}
&\sum_{\substack{A,B,C,D \in \mathcal{M} \\ \degree AB < \degree Q \\ \degree CD < \degree Q \\ AC \equiv BD (\modulus Q)  \\ AC \neq BD}} \frac{p \big( \degree A, \degree B, \degree C , \degree D, \degree Q \big)}{\lvert ABCD \rvert^{\frac{1}{2}}} \\
\ll_p &(\degree Q)^{d} \sum_{\substack{A,B,C,D \in \mathcal{M} \\ \degree AB < \degree Q \\ \degree CD < \degree Q \\ AC \equiv BD (\modulus Q)  \\ AC \neq BD}} \frac{1}{\lvert ABCD \rvert^{\frac{1}{2}}} 
= (\degree Q)^{d} \sum_{0 \leq z_1 , z_2 < \degree Q} q^{-\frac{z_1 + z_2}{2}} \sum_{\substack{A,B,C,D \in \mathcal{M} \\ \degree AB = z_1 \\ \degree CD = z_2 \\ AC \equiv BD (\modulus Q)  \\ AC \neq BD}} 1 .
\end{split}
\end{align}

Now, Lemma 7.10 from \cite{FourPowMeanDirLFuncFuncField} tells us that for non-negative integers $z_1 , z_2$ we have
\begin{align}\label{Reference for off-diagonal terms from primitive moments paper}
\sum_{\substack{A,B,C,D \in \mathcal{M} \\ \degree AB = z_1 \\ \degree CD = z_2 \\ AC \equiv BD (\modulus Q)  \\ AC \neq BD}} 1
\begin{cases}
\ll_{\epsilon } \frac{1}{\lvert Q \rvert} \big( q^{z_1} q^{z_2} \big)^{1+\epsilon} &\text{ if $z_1 + z_2 \leq \frac{19}{10} \degree Q$} \\
\ll \frac{1}{\phi (Q)} q^{z_1} q^{z_2} (z_1 + z_2 )^3 &\text{ if $z_1 + z_2 > \frac{19}{10} \degree Q$} .
\end{cases}
\end{align}
Hence, for $\epsilon < \frac{1}{38}$ we have
\begin{align*}
&\sum_{0 \leq z_1 , z_2 < \degree Q} q^{-\frac{z_1 + z_2}{2}} \sum_{\substack{A,B,C,D \in \mathcal{M} \\ \degree AB = z_1 \\ \degree CD = z_2 \\ AC \equiv BD (\modulus Q)  \\ AC \neq BD}} 1 \\
\ll &\frac{1}{\lvert Q \rvert} \sum_{\substack{0 \leq z_1 , z_2 < \degree Q \\ z_1 + z_2 \leq \frac{19}{10} \degree Q}} \big( q^{\frac{1}{2} + \epsilon} \big)^{z_1 + z_2} 
+ \frac{1}{\phi (Q)} \sum_{\substack{0 \leq z_1 , z_2 < \degree Q \\  z_1 + z_2 > \frac{19}{10} \degree Q}} q^{\frac{z_1 + z_2}{2}} (z_1 + z_2)^3  \\
\ll &\frac{\lvert Q \rvert}{\phi (Q)} (\degree Q)^3 
\ll (\degree Q)^3 
\end{align*}
as $\degree Q \longrightarrow \infty$. The result follows by applying this to (\ref{Off-diagonal z_1, z_2 split}).
\end{proof}

\begin{remark} \label{Remark, Addressing Tamam Error}
In her paper, Tamam \cite[Lemma 8.5]{FourthMoment_Tamam} states a similar result as (\ref{Reference for off-diagonal terms from primitive moments paper}) above. However, in her proof she claims that $d(N) \ll \deg N$, which is not the case. Addressing this is non-trivial and was done in \cite{FourPowMeanDirLFuncFuncField}, as stated above.
\end{remark}

\begin{lemma} \label{Lemma, all non-trv char remainder}
Let $p \big( \degree A, \degree B, \degree C , \degree D, \degree Q \big)$ be a finite polynomial of degree $d$. Then,
\begin{align*}
\frac{1}{\phi (Q)} \sum_{\substack{A,B,C,D \in \mathcal{M} \\ \degree AB < \degree Q \\ \degree CD < \degree Q}} \frac{p \big( \degree A, \degree B, \degree C , \degree D, \degree Q \big)}{\lvert ABCD \rvert^{\frac{1}{2}}}
\ll_p (\degree Q)^{d+2} .
\end{align*}
\end{lemma}

\begin{proof}
Because $\degree AB , \degree CD < \degree Q$, we have that
\begin{align*}
p \big( \degree A, \degree B, \degree C , \degree D, \degree Q \big)
\ll_p (\degree Q)^{d} .
\end{align*}
Hence,
\begin{align*}
&\frac{1}{\phi (Q)} \sum_{\substack{A,B,C,D \in \mathcal{M} \\ \degree AB < \degree Q \\ \degree CD < \degree Q}} \frac{p \big( \degree A, \degree B, \degree C , \degree D, \degree Q \big)}{\lvert ABCD \rvert^{\frac{1}{2}}} \\
\ll_p &\frac{(\degree Q)^{d}}{\phi (Q)} \bigg( \sum_{\substack{A,B \in \mathcal{M} \\ \degree AB < \degree Q}} \frac{1}{\lvert AB \rvert^{\frac{1}{2}}} \bigg) \bigg(  \sum_{\substack{C,D \in \mathcal{M} \\ \degree CD < \degree Q}} \frac{1}{\lvert CD \rvert^{\frac{1}{2}}} \bigg) \\
= &\frac{(\degree Q)^{d}}{\phi (Q)} \bigg( \sum_{\substack{n,m \geq 0 \\ n+m < \degree Q}} q^{\frac{m+n}{2}} \bigg)^2 
\ll (\degree Q)^{d+2} .
\end{align*}
\end{proof}

From Lemmas \ref{Lemma, all nontriv char split to diag, off diag, remainder} to \ref{Lemma, all non-trv char remainder} we can deduce the following:

\begin{proposition} \label{Proposition, sum of all char of p_1 p_2}
Let $Q \in \mathcal{M}$ be prime, and let $p_1 \big( \degree A, \degree B, \degree Q \big)$ and $p_2 \big( \degree A, \degree B, \degree Q \big)$ be finite homogeneous polynomials of degree $d_1$ and $d_2$, respectively. Then,
\begin{align*}
&\frac{1}{\phi (Q)} \sum_{\substack{\chi \modulus Q \\ \chi \neq \chi_0}} \bigg( \sum_{\substack{ A,B \in \mathcal{M} \\ \degree AB < \degree Q}} \frac{p_1 \; \chi (A) \charinv\chi (B)}{\lvert AB \rvert^{\frac{1}{2}}} \bigg) \bigg( \sum_{\substack{ C,D \in \mathcal{M} \\ \degree CD < \degree Q}} \frac{p_2 \; \chi (C) \charinv\chi (D)}{\lvert CD \rvert^{\frac{1}{2}}} \bigg) \\
= &(1- q^{-1}) (\degree Q)^{d_1 + d_2 +4} \int_{\substack{a_1 , a_2 , a_3 , a_4 \geq 0 \\ 2a_1 + a_3 + a_4 < 1 \\ 2a_2 + a_3 + a_4 < 1}}  p_1 \big( a_1 + a_3 , a_1 + a_4 , 1 \big) p_2 \big( a_2 + a_4 , a_2 + a_3 , 1 \big) \mathrm{d} a_1 \mathrm{d} a_2 \mathrm{d} a_3 \mathrm{d} a_4 \\
& + O_{p_1 , p_2 } \big( (\degree Q)^{d_1 + d_2 +3} \big) .
\end{align*}
\end{proposition}

Similarly, the following can be proved:

\begin{proposition} \label{Proposition, sum of even char of p_1 p_2}
Let $Q \in \mathcal{M}$ be prime, and let $p_1 \big( \degree A, \degree B, \degree Q \big)$ and $p_2 \big( \degree A, \degree B, \degree Q \big)$ be finite homogeneous polynomials of degree $d_1$ and $d_2$, respectively. Then,
\begin{align*}
&\frac{1}{\phi (Q)} \sum_{\substack{\chi \modulus Q \\ \chi \text{ even} \\ \chi \neq \chi_0}} \bigg( \sum_{\substack{ A,B \in \mathcal{M} \\ \degree AB < \degree Q}} \frac{p_1 \; \chi (A) \charinv\chi (B)}{\lvert AB \rvert^{\frac{1}{2}}} \bigg) \bigg( \sum_{\substack{ C,D \in \mathcal{M} \\ \degree CD < \degree Q}} \frac{p_2 \; \chi (C) \charinv\chi (D)}{\lvert CD \rvert^{\frac{1}{2}}} \bigg) \\
= &q^{-1} (\degree Q)^{d_1 + d_2 +4} \int_{\substack{a_1 , a_2 , a_3 , a_4 \geq 0 \\ 2a_1 + a_3 + a_4 < 1 \\ 2a_2 + a_3 + a_4 < 1}}  p_1 \big( a_1 + a_3 , a_1 + a_4 , 1 \big) p_2 \big( a_2 + a_4 , a_2 + a_3 , 1 \big) \mathrm{d} a_1 \mathrm{d} a_2 \mathrm{d} a_3 \mathrm{d} a_4 \\
& + O_{p_1 , p_2 } \big( (\degree Q)^{d_1 + d_2+3} \big) .
\end{align*}
\end{proposition}

The proof of Proposition \ref{Proposition, sum of even char of p_1 p_2} is similar to the proof of Proposition \ref{Proposition, sum of all char of p_1 p_2}. From \cite{FourPowMeanDirLFuncFuncField}, we use Lemma 7.11 instead of Lemma 7.10. \\

We can similarly prove the following:

\begin{proposition} \label{Proposition, short sum of even/all char of p_1 p_2}
Let $Q \in \mathcal{M}$ be prime, let $p_1 \big( \degree A, \degree B, \degree Q \big)$ and $p_2 \big( \degree A, \degree B, \degree Q \big)$ be finite homogeneous polynomials of degree $d_1$ and $d_2$, respectively, and let $a \in \{ 0,1,2,3 \}$. Then,
\begin{align*}
\frac{1}{\phi (Q)} \sum_{\substack{\chi \modulus Q \\ \chi \neq \chi_0}} \bigg( \sum_{\substack{ A,B \in \mathcal{M} \\ \degree AB = \degree Q -a}} \frac{p_1 \; \chi (A) \charinv\chi (B)}{\lvert AB \rvert^{\frac{1}{2}}} \bigg) \bigg( \sum_{\substack{ C,D \in \mathcal{M} \\ \degree CD = \degree Q -a}} \frac{p_2 \; \chi (C) \charinv\chi (D)}{\lvert CD \rvert^{\frac{1}{2}}} \bigg)
= O_{p_1 , p_2 } \big( (\degree Q)^{d_1 + d_2 +3} \big) ,
\end{align*}
and
\begin{align*}
\frac{1}{\phi (Q)} \sum_{\substack{\chi \modulus Q \\ \chi \text{ even} \\ \chi \neq \chi_0}} \bigg( \sum_{\substack{ A,B \in \mathcal{M} \\ \degree AB = \degree Q -a}} \frac{p_1 \; \chi (A) \charinv\chi (B)}{\lvert AB \rvert^{\frac{1}{2}}} \bigg) \bigg( \sum_{\substack{ C,D \in \mathcal{M} \\ \degree CD = \degree Q -a}} \frac{p_2 \; \chi (C) \charinv\chi (D)}{\lvert CD \rvert^{\frac{1}{2}}} \bigg)
= O_{p_1 , p_2 } \big( (\degree Q)^{d_1 + d_2+3} \big) .
\end{align*}
\end{proposition}


\section{Fourth Moments}

We are now equipped to prove the fourth moment result.

\begin{proof}[Proof of Theorem \ref{fourth_moment_theorem}]
We have that
\begin{align}
\begin{split} \label{L^(k) L^(l) split into odd and even char}
&\frac{1}{\phi (Q)} \sum_{\substack{\chi \modulus Q \\ \chi \neq \chi_0}} \Big\lvert L^{(k)} \Big( \frac{1}{2} , \chi \Big) \Big\rvert^{2} \Big\lvert L^{(l)} \Big( \frac{1}{2} , \chi \Big) \Big\rvert^{2} \\
&= \frac{1}{\phi (Q)} \sum_{\substack{\chi \modulus Q \\ \chi \text{ odd}}} \Big\lvert L^{(k)} \Big( \frac{1}{2} , \chi \Big) \Big\rvert^{2} \Big\lvert L^{(l)} \Big( \frac{1}{2} , \chi \Big) \Big\rvert^{2}
+ \frac{1}{\phi (Q)} \sum_{\substack{\chi \modulus Q \\ \chi \text{ even} \\ \chi \neq \chi_0}} \Big\lvert L^{(k)} \Big( \frac{1}{2} , \chi \Big) \Big\rvert^{2} \Big\lvert L^{(l)} \Big( \frac{1}{2} , \chi \Big) \Big\rvert^{2} .
\end{split}
\end{align}

Using Proposition \ref{Proposition, odd character L(1/2 ,chi) shorter sum}, we have for the first term on the RHS that
\begin{align}
\begin{split} \label{Odd char L^(k) L^(l) split}
&\frac{1}{\phi (Q)} \frac{1}{(\log q)^{2k+2l}} \sum_{\substack{\chi \modulus Q \\ \chi \text{ odd}}} \Big\lvert L^{(k)} \Big( \frac{1}{2} , \chi \Big) \Big\rvert^{2} \Big\lvert L^{(l)} \Big( \frac{1}{2} , \chi \Big) \Big\rvert^{2} \\
&= \frac{1}{\phi (Q)} \sum_{\substack{\chi \modulus Q \\ \chi \text{ odd}}}
\Bigg( \sum_{\substack{A,B \in \mathcal{M} \\ \degree AB < \degree Q}} \frac{ \Big( f_{k} + g_{O,k} \Big) \chi (A) \charinv\chi (B)}{\lvert AB \rvert^{\frac{1}{2}}} + \sum_{\substack{A,B \in \mathcal{M} \\ \degree AB = \degree Q -1}} \frac{ h_{O,k} \chi (A) \charinv\chi (B)}{\lvert AB \rvert^{\frac{1}{2}}} \Bigg) \\
&\quad \quad \quad \quad \quad \quad \quad \quad \cdot \Bigg( \sum_{\substack{C,D \in \mathcal{M} \\ \degree CD < \degree Q}} \frac{ \Big( f_{l} + g_{O,l} \Big) \chi (C) \charinv\chi (D)}{\lvert CD \rvert^{\frac{1}{2}}} + \sum_{\substack{C,D \in \mathcal{M} \\ \degree CD = \degree Q -1}} \frac{ h_{O,l} \chi (C) \charinv\chi (D)}{\lvert CD \rvert^{\frac{1}{2}}} \Bigg) .
\end{split}
\end{align}
By using Propositions \ref{Proposition, sum of all char of p_1 p_2} and \ref{Proposition, sum of even char of p_1 p_2}, we have that
\begin{align*}
&\frac{1}{\phi (Q)} \sum_{\substack{\chi \modulus Q \\ \chi \text{ odd}}}
\Bigg( \sum_{\substack{A,B \in \mathcal{M} \\ \degree AB < \degree Q}} \frac{ \Big( f_{k} + g_{O,k} \Big) \chi (A) \charinv\chi (B)}{\lvert AB \rvert^{\frac{1}{2}}}  \Bigg)
\Bigg( \sum_{\substack{C,D \in \mathcal{M} \\ \degree CD < \degree Q}} \frac{ \Big( f_{l} + g_{O,l} \Big) \chi (C) \charinv\chi (D)}{\lvert CD \rvert^{\frac{1}{2}}} \Bigg) \\
&= \frac{1}{\phi (Q)} \sum_{\substack{\chi \modulus Q \\ \chi \neq \chi_0}}
\Bigg( \sum_{\substack{A,B \in \mathcal{M} \\ \degree AB < \degree Q}} \frac{ \Big( f_{k} + g_{O,k} \Big) \chi (A) \charinv\chi (B)}{\lvert AB \rvert^{\frac{1}{2}}}  \Bigg)
\Bigg( \sum_{\substack{C,D \in \mathcal{M} \\ \degree CD < \degree Q}} \frac{ \Big( f_{l} + g_{O,l} \Big) \chi (C) \charinv\chi (D)}{\lvert CD \rvert^{\frac{1}{2}}} \Bigg) \\
& \quad - \frac{1}{\phi (Q)} \sum_{\substack{\chi \modulus Q \\ \chi \text{ even} \\ \chi \neq \chi_0}}
\Bigg( \sum_{\substack{A,B \in \mathcal{M} \\ \degree AB < \degree Q}} \frac{ \Big( f_{k} + g_{O,k} \Big) \chi (A) \charinv\chi (B)}{\lvert AB \rvert^{\frac{1}{2}}}  \Bigg)
\Bigg( \sum_{\substack{C,D \in \mathcal{M} \\ \degree CD < \degree Q}} \frac{ \Big( f_{l} + g_{O,l} \Big) \chi (C) \charinv\chi (D)}{\lvert CD \rvert^{\frac{1}{2}}} \Bigg) \\
&= (1- 2q^{-1}) (\degree Q)^{2k+2l+4} \int_{\substack{a_1 , a_2 , a_3 , a_4 \geq 0 \\ 2a_1 + a_3 + a_4 < 1 \\ 2a_2 + a_3 + a_4 < 1}}  f_k \big( a_1 + a_3 , a_1 + a_4 , 1 \big) f_l \big( a_2 + a_4 , a_2 + a_3 , 1 \big) \mathrm{d} a_1 \mathrm{d} a_2 \mathrm{d} a_3 \mathrm{d} a_4 \\
& \quad + O_{k,l} \Big( (\degree Q)^{2k+2l+3} \Big) 
\end{align*}
as $\degree Q \longrightarrow \infty$. Strictly speaking, Propositions \ref{Proposition, sum of all char of p_1 p_2} and \ref{Proposition, sum of even char of p_1 p_2} require that the polynomials $f_{k} + g_{O,k}$ and $f_{l} + g_{O,l}$ are homogeneous, which is not the case. However, these polynomials can be written as sums of homogeneous polynomials, with the terms of highest degree being $f_{k}$ and $f_{l}$, respectively. We can then apply the propositions term-by-term to obtain the result above. \\

We now have the main term of (\ref{Odd char L^(k) L^(l) split}). Indeed, for the remaining terms we can apply the Cauchy-Schwarz inequality and Propositions \ref{Proposition, sum of all char of p_1 p_2}, \ref{Proposition, sum of even char of p_1 p_2}, and \ref{Proposition, short sum of even/all char of p_1 p_2} to see that they are equal to $O_{k,l} \Big( (\degree Q)^{2k+2l+\frac{7}{2}} \Big)$. Hence,
\begin{align}
\begin{split} \label{Odd char L^(k) L^(l) split, main and LO terms}
&\frac{1}{\phi (Q)} \frac{1}{(\log q)^{2k+2l}} \sum_{\substack{\chi \modulus Q \\ \chi \text{ odd}}} \Big\lvert L^{(k)} \Big( \frac{1}{2} , \chi \Big) \Big\rvert^{2} \Big\lvert L^{(l)} \Big( \frac{1}{2} , \chi \Big) \Big\rvert^{2} \\
&= (1- 2q^{-1}) (\degree Q)^{2k+2l+4} \int_{\substack{a_1 , a_2 , a_3 , a_4 \geq 0 \\ 2a_1 + a_3 + a_4 < 1 \\ 2a_2 + a_3 + a_4 < 1}}  f_k \big( a_1 + a_3 , a_1 + a_4 , 1 \big) f_l \big( a_2 + a_4 , a_2 + a_3 , 1 \big) \mathrm{d} a_1 \mathrm{d} a_2 \mathrm{d} a_3 \mathrm{d} a_4 \\
& \quad + O_{k,l} \Big( (\degree Q)^{2k+2l+\frac{7}{2}} \Big) 
\end{split}
\end{align}
as $\degree Q \longrightarrow \infty$. \\

We now look at the second term on the RHS of (\ref{L^(k) L^(l) split into odd and even char}). By using Proposition \ref{Proposition, even character L(1/2 ,chi) shorter sum} and similar means as those used to deduce (\ref{Odd char L^(k) L^(l) split, main and LO terms}), we can show for all non-negative integers $i,j$ that
\begin{align*}
&\frac{1}{\phi (Q)} \frac{1}{(\log q)^{2i+2j}} \frac{1}{(q^{\frac{1}{2}} -1)^4} \sum_{\substack{\chi \modulus Q \\ \chi \text{ even} \\ \chi \neq \chi_0}} \Big\lvert \hat{L}^{(i)} \Big( \frac{1}{2} , \chi \Big) \Big\rvert^{2} \Big\lvert \hat{L}^{(j)} \Big( \frac{1}{2} , \chi \Big) \Big\rvert^{2} \\
&= \frac{q^{-1}}{(q^{\frac{1}{2}} -1)^4} (\degree Q)^{2i+2j+4} \int_{\substack{a_1 , a_2 , a_3 , a_4 \geq 0 \\ 2a_1 + a_3 + a_4 < 1 \\ 2a_2 + a_3 + a_4 < 1}}  f_i \big( a_1 + a_3 , a_1 + a_4 , 1 \big) f_j \big( a_2 + a_4 , a_2 + a_3 , 1 \big) \mathrm{d} a_1 \mathrm{d} a_2 \mathrm{d} a_3 \mathrm{d} a_4 \\
& \quad + O_{i,j} \Big( (\degree Q)^{2i+2j+\frac{7}{2}} \Big) 
\end{align*}
as $\degree Q \longrightarrow \infty$. Using Proposition \ref{Proposition, L in terms of L hat} and the Cauchy-Schwarz inequality, we obtain that
\begin{align}
\begin{split} \label{Even char L^(k) L^(l) split, main and LO terms}
&\frac{1}{\phi (Q)} \frac{1}{(\log q)^{2k+2l}} \sum_{\substack{\chi \modulus Q \\ \chi \text{ even} \\ \chi \neq \chi_0}} \Big\lvert L^{(k)} \Big( \frac{1}{2} , \chi \Big) \Big\rvert^{2} \Big\lvert L^{(l)} \Big( \frac{1}{2} , \chi \Big) \Big\rvert^{2} \\
&= q^{-1} (\degree Q)^{2k+2l+4} \int_{\substack{a_1 , a_2 , a_3 , a_4 \geq 0 \\ 2a_1 + a_3 + a_4 < 1 \\ 2a_2 + a_3 + a_4 < 1}}  f_k \big( a_1 + a_3 , a_1 + a_4 , 1 \big) f_l \big( a_2 + a_4 , a_2 + a_3 , 1 \big) \mathrm{d} a_1 \mathrm{d} a_2 \mathrm{d} a_3 \mathrm{d} a_4 \\
& \quad + O_{k,l} \Big( (\degree Q)^{2k+2l+\frac{7}{2}} \Big) 
\end{split}
\end{align}
as $\degree Q \longrightarrow \infty$. \\

The proof folllows from (\ref{L^(k) L^(l) split into odd and even char}), (\ref{Odd char L^(k) L^(l) split, main and LO terms}), (\ref{Even char L^(k) L^(l) split, main and LO terms}).
\end{proof}

We now proceed to prove Theorem \ref{Fourth moment coefficients asymptotic behaviour}.

\begin{lemma} \label{(1- x/m )^m bound}
Let $m$ be a positive integer. For all non-negative $x$ we have that
\begin{align*}
\Big(1- \frac{x}{m} \Big)^m \leq e^{-x} ,
\end{align*}
and for all $x \in [0, 2 m^{\frac{1}{3}}]$ we have that
\begin{align*}
\Big(1- \frac{x}{m} \Big)^m \geq e^{-x} e^{\frac{-4}{m^{\frac{1}{3}} - 2 m^{-\frac{1}{3}}}} .
\end{align*}
\end{lemma}

\begin{proof}
By using the Taylor seies for $\log$ we have that
\begin{align*}
\log \Big( \big(1- \frac{x}{m} \big)^m \Big)
= -x - \frac{x^2}{2m} - \frac{x^3}{3m^2} - \frac{x^4}{4m^3} - \ldots .
\end{align*}
Clearly, the RHS is $\leq - x$, which proves the first inequality. For the second inequality we use the bounds on $x$ to obtain that
\begin{align*}
\frac{x^2}{2m} + \frac{x^3}{3m^2} + \frac{x^4}{4m^3} + \ldots
\leq \frac{x^2}{m} \sum_{i=0}^{\infty} \Big( \frac{x}{m} \Big)^i
= \frac{x^2}{m} \bigg( \frac{1}{1-\frac{x}{m}} \bigg)
\leq \bigg( \frac{4}{m^{\frac{1}{3}} -2m^{-\frac{1}{3}}} \bigg) ,
\end{align*}
from which the result follows.
\end{proof}

\begin{proof}[Proof of Theorem \ref{Fourth moment coefficients asymptotic behaviour}]
Let us expand the brackets in (\ref{4th moment, mth derivative main coeff limit}) and multiply by $m^4$. One of the terms is the following:
\begin{align*}
&m^4 \int_{\substack{a_1 , a_2 , a_3 , a_4 \geq 0 \\ 2a_1 + a_3 + a_4 < 1 \\ 2a_2 + a_3 + a_4 < 1}} (1- a_1 - a_3 )^m (1 - a_1 - a_4 )^m (1- a_2 - a_3 )^m (1 - a_2 - a_4 )^m \mathrm{d} a_1 \mathrm{d} a_2 \mathrm{d} a_3 \mathrm{d} a_4 \\ 
= &\int_{\substack{a_1 , a_2 , a_3 , a_4 \geq 0 \\ 2a_1 + a_3 + a_4 < m \\ 2a_2 + a_3 + a_4 < m }} \Big(1 - \frac{a_1 + a_3}{m} \Big)^m \Big(1 - \frac{a_1 + a_4}{m} \Big)^m \Big(1 - \frac{a_2 + a_3}{m} \Big)^m \Big(1 - \frac{a_2 + a_4}{m} \Big)^m \mathrm{d} a_1 \mathrm{d} a_2 \mathrm{d} a_3 \mathrm{d} a_4 ,
\end{align*}
where we have used the substitutions $a_i \to \frac{a_i}{m}$. On one hand, by using Lemma \ref{(1- x/m )^m bound}, we have that
\begin{align*}
&\int_{\substack{a_1 , a_2 , a_3 , a_4 \geq 0 \\ 2a_1 + a_3 + a_4 < m \\ 2a_2 + a_3 + a_4 < m }} \Big(1 - \frac{a_1 + a_3}{m} \Big)^m \Big(1 - \frac{a_1 + a_4}{m} \Big)^m \Big(1 - \frac{a_2 + a_3}{m} \Big)^m \Big(1 - \frac{a_2 + a_4}{m} \Big)^m \mathrm{d} a_1 \mathrm{d} a_2 \mathrm{d} a_3 \mathrm{d} a_4 \\
\geq &\int_{\substack{0 \leq a_1 , a_2 , a_3 , a_4 \leq m^{\frac{1}{3}} \\ 2a_1 + a_3 + a_4 < m \\ 2a_2 + a_3 + a_4 < m }} \Big(1 - \frac{a_1 + a_3}{m} \Big)^m \Big(1 - \frac{a_1 + a_4}{m} \Big)^m \Big(1 - \frac{a_2 + a_3}{m} \Big)^m \Big(1 - \frac{a_2 + a_4}{m} \Big)^m \mathrm{d} a_1 \mathrm{d} a_2 \mathrm{d} a_3 \mathrm{d} a_4 \\
\geq  & e^{\frac{-16}{m^{\frac{1}{3}} - 2 m^{-\frac{1}{3}}}} \int_{0 \leq a_1 , a_2 , a_3 , a_4 \leq m^{\frac{1}{3}}} e^{-2(a_1 + a_2 + a_3 + a_4)} \mathrm{d} a_1 \mathrm{d} a_2 \mathrm{d} a_3 \mathrm{d} a_4
\longrightarrow \frac{1}{16} 
\end{align*}
On the other hand, by the same lemma, we have that
\begin{align*}
&\int_{\substack{a_1 , a_2 , a_3 , a_4 \geq 0 \\ 2a_1 + a_3 + a_4 < m \\ 2a_2 + a_3 + a_4 < m }} \Big(1 - \frac{a_1 + a_3}{m} \Big)^m \Big(1 - \frac{a_1 + a_4}{m} \Big)^m \Big(1 - \frac{a_2 + a_3}{m} \Big)^m \Big(1 - \frac{a_2 + a_4}{m} \Big)^m \mathrm{d} a_1 \mathrm{d} a_2 \mathrm{d} a_3 \mathrm{d} a_4 \\
\leq &\int_{\substack{0 \leq a_1 , a_2 , a_3 , a_4 \leq m}} e^{-2(a_1 + a_2 + a_3 + a_4)} \mathrm{d} a_1 \mathrm{d} a_2 \mathrm{d} a_3 \mathrm{d} a_4
\longrightarrow \frac{1}{16} 
\end{align*}

So, we see that
\begin{align} \label{4th moment, mth derivative main coeff limit. Main term}
m^4 \int_{\substack{a_1 , a_2 , a_3 , a_4 \geq 0 \\ 2a_1 + a_3 + a_4 < 1 \\ 2a_2 + a_3 + a_4 < 1}} (1- a_1 - a_3 )^m (1 - a_1 - a_4 )^m (1- a_2 - a_3 )^m (1 - a_2 - a_4 )^m \mathrm{d} a_1 \mathrm{d} a_2 \mathrm{d} a_3 \mathrm{d} a_4 
\longrightarrow \frac{1}{16}
\end{align}
as $m \longrightarrow \infty$.\\

Now, after we expanded the brackets in (\ref{4th moment, mth derivative main coeff limit}) and multiplied by $m^4$, there were other terms. These can be seen to tend to $0$ as $m \longrightarrow \infty$. We prove one case below; the rest are similar.

\begin{align*}
&\int_{\substack{a_1 , a_2 , a_3 , a_4 \geq 0 \\ 2a_1 + a_3 + a_4 < 1 \\ 2a_2 + a_3 + a_4 < 1}} (1- a_1 - a_3 )^m (1 - a_1 - a_4 )^m (a_2 + a_3 )^m (a_2 + a_4 )^m \mathrm{d} a_1 \mathrm{d} a_2 \mathrm{d} a_3 \mathrm{d} a_4 \\
\leq &\int_{\substack{a_1 , a_2 , a_3 , a_4 \geq 0 \\ 2a_1 + a_3 + a_4 < 1 \\ 2a_2 + a_3 + a_4 < 1}} (a_2 + a_3 )^m (a_2 + a_4 )^m 
\ll \frac{1}{4^m} ,
\end{align*}
where we have used the following: The maximum value that $(a_2 + a_3 )(a_2 + a_4 )$ can take subject to the conditions in the integral is at most equal to the maximum value that $f(x,y) := xy$ can take subject to the conditions $x,y \geq 0$ and $x+y < 1$. By plotting this range and looking at contours of $f(x,y)$ we can see that the maximum value is $\frac{1}{4}$. The result follows. 
\end{proof}

\bibliography{YiasemidesBibliography1}{}
\bibliographystyle{bibstyle1}

\end{document}